\documentclass{amsart}
\usepackage{hyperref}
\usepackage{verbatim}
\usepackage{nccmath}
\usepackage{graphicx}
\usepackage{enumitem}
\usepackage{mathtools}
\mathtoolsset{showonlyrefs}
\allowdisplaybreaks
\def \R{\mathbb R}
\def \B{\mathbb B}
\def \N{\mathbb N}

\def \C{\mathbb C}
\def \im{\operatorname{Im}}
\def \id{\operatorname{id}}
\newcommand \vol[2][3]{\left|#2\right|_{#1}}
\newtheorem{theorem}{Theorem}
\newtheorem{lemma}{Lemma}
\newtheorem{proposition}{Proposition}

\title[The Barker--Larman problem in dimension $4$]{An approximate counterexample to the Barker--Larman problem in dimension $4$}
\author{J. Haddad}

\begin{document}
\begin{abstract}
	The Barker--Larman problem asks if a convex body $K \subseteq \R^n$ containing the Euclidean ball $\B_n$, such that all the sections of $K$ by hyperplanes tangent to $\B_n$ have constant $(n-1)$-dimensional volume, must necessarily be a Euclidean ball.

	In this paper we show a result pointing to a negative answer in dimension $4$.
	Taking $\lambda_0 = 4 \sqrt{3}\pi$ and any $N \in \N$, we obtain the existence of a family of convex bodies $K_{\lambda,N}$ with $\lambda \in (\lambda_0-r_N, \lambda_0 + r_N)$, such that the sections of $K_{\lambda,N}$ by hyperplanes tangent to the Euclidean ball, have area within $c |\lambda - \lambda_0|^{N+1}$ of $\lambda$, while the difference between outradius and inradius of $K_{\lambda,N}$ is larger than $C |\lambda - \lambda_0|$.

	The bodies $K_{\lambda,N}$ are constructed via radial functions as
	\[\rho_{K_{\lambda,N}}(t) = \cos\left( \sum_{n=0}^N \frac{(\lambda-\lambda_0)^n}{n!} \varphi_n(t) \right)^{-1},\]
	where $t \in [0,2\pi), \lambda \in \R$ and $\varphi_n$ are trigonometric polynomials that can be computed explicitly.
	The convergence of the inner power series when $N \to \infty$ (which is left open) would imply a negative answer to the Barker--Larman problem in dimension $4$.

	As an example we obtain a convex body whose outradius and inradius differ by more than $0.176$, and the area of the sections oscillate by less than $3 \times 10^{-7}$.
\end{abstract}
\maketitle
\section{Introduction}
\label{sec_intro}

In 2001 Barker and Larman formulated the following question which is now an important open problem in geometric tomography:
Let $K \subseteq \R^n$ be a convex body (a compact and convex set) containing the centered unit Euclidean ball $\B_n$ in its interior. Assume that for every hyperplane $H$ tangent to $\B_n$,
\begin{equation}
	\label{eq_equalsections}
	\vol[n-1]{K \cap H} = \lambda,
\end{equation}
where $\lambda > 0$ is fixed and independent of $H$, and $\vol[k]{ \ \cdot \ }$ denotes $k$-dimensional volume.
Is it true that $K$ has to be itself a Euclidean ball?

This innocent--looking problem is still open in general, although some particular cases and variations of this problem were solved recently.
We will refer to the sections $K \cap H$ where $H$ is tangent to $\B_n$ as the $1$-sections, since the hyperplanes $H$ are at distance $1$ from the origin.

The case $n=2$ had been settled in the affirmative by Santaló in \cite{santalo1951two}, back in 1951.
Clearly, if $K$ is a Euclidean ball, then its radius is determined by $\lambda$.
Yaskin and Yaskina in \cite{yaskin2015thick} proved $K$ is a Euclidean ball in dimension $3$, under the additional assumption that the volume of $K$ equals the volume of the ball whose $1$-sections have area $\lambda$.

Similar results of determination of convex bodies by $1$-sections were established in \cite{yaskin2011unique, alfonseca2025characterizations, alfonseca2025croft, gardner2012problem, kurusa2015characterizations, nazarov2014asymmetric, yaskin2017non}.

More recently the author together with D. Ryabogin showed in \cite{haddad2026convex} that the answer in dimension $n=4$ is positive for symmetric convex bodies of revolution, and $\lambda \in \Lambda \subseteq \R_+$, where $\Lambda$ is a certain set determined by arithmetic properties.
Roughly speaking, consider $\alpha = \arctan((\lambda/\omega_3)^{1/3})$ where $\omega_k$ denotes the $k$-dimensional volume of the Euclidean ball $\B_k$.
Then the condition $\lambda \in \Lambda$ holds when $\frac{\alpha}{\pi}$ cannot be approximated by rational numbers very rapidly (see \cite[Definition 8]{haddad2026convex}).
The conditions imposed in \cite{haddad2026convex} may seem very restrictive for a positive answer, but this is actually very close to the best possible result in the positive direction.
The arguments in \cite{haddad2026convex} break down completely if $\frac{\alpha}{\pi}$ is a rational number, and the set $\Lambda$ is nowhere dense. This means that there is still room for counterexamples for many values of $\lambda \in \R \setminus \Lambda$. In this paper we focus in the case $\alpha_0=\frac \pi 3$, or equivalently $\lambda_0 = 4 \sqrt{3} \pi$, and try to find a one-parameter family of counterexamples for $\lambda$ near $\lambda_0$.

For any symmetric convex body of revolution $K$ and $x \in [0,\pi]$, call $A_K(x)$ the $3$-dimensional volume of any section of $K$ by a hyperplane tangent to $\B_4$, whose normal forms an angle $x$ with the axis of revolution.

In order to measure how much a convex body $K$ differs from a Euclidean ball, we shall use $R(K)$ and $r(K)$, which denote its outradius and inradius, respectively.
In this paper we prove the following:

\begin{theorem}
	\label{thm_approximate}
	Let $\lambda_0 = 4 \sqrt{3}\pi$.
	For any $N \in \N$ there is $r_N > 0$ and a $1$-parameter family of symmetric convex bodies of revolution $K_{\lambda, N}$ for $\lambda \in (\lambda_0-r_N, \lambda_0 + r_N)$, such that for every $x \in (0,\pi)$ and $\lambda \in (\lambda_0-r_N, \lambda_0 + r_N)$,
	\[| \lambda - A_{K_{\lambda,N}}(x)| \leq c \left|\frac{\lambda-\lambda_0}{r_N} \right|^{N+1},\]
	while
	\[ |R(K_{\lambda,N}) - r(K_{\lambda,N})| \geq c |\lambda - \lambda_0|,\]
	where $c>0$ is a universal constant.

	For every $N$, the set $K_{\lambda_0, N}$ is the Euclidean ball of radius $2$.
\end{theorem}

The convex bodies in Theorem \ref{thm_approximate} are roughly constructed in the following way:
Let $h:\R \to \R$ be a smooth increasing function satisfying $h(t + 2 \pi) = h(t) + 2\pi$.
The function
\begin{equation}
	\label{eq_intro_curve}
	\frac 1 {\cos(t - h(t))} (\cos(t), \sin(t)),\quad t \in [0, 2\pi),
\end{equation}
parametrizes a smooth curve in the plane, which is, under additional conditions on $h$, the boundary of a convex body $L_h \subseteq \R^2$ symmetric with respect to the horizontal axis.
Thus, it can be used to define a convex body of revolution
\[K_h = \{(t, y) \in \R \times \R^3 : (t, \|y\|_2) \in L_h\} \subseteq \R^4.\]
The curve \eqref{eq_intro_curve} is a very convenient way of parametrizing convex bodies of revolution with a good control on its $1$-sections.

The problem of finding a function $h$ for which $K_h$ has $1$-sections of constant volume has the characteristics of a bifurcation problem.
For each $\lambda>0$ there is a trivial solution
\[
	h_\lambda^{\text{(trivial)}}(t) = t - \arctan((\lambda/\omega_3)^{1/3}),
\]
for which $K_{h_\lambda^{\text{(trivial)}}}$ is the centered Euclidean ball whose $1$-sections have volume $\lambda$ (it is the ball of radius $\sqrt{1+(\lambda/\omega_3)^{2/3}}$). This set of solutions $h_\lambda^{\text{(trivial)}}$ forms the ``trivial branch'' of solutions.
We fix $\lambda_0 = 4 \sqrt{3}\pi$ and search a non-trivial branch of functions $h_\lambda$ for $\lambda$ close to $\lambda_0$, which crosses the trivial branch at $\lambda = \lambda_0$.

The classical bifurcation theorem of Crandall-Rabinowitz cannot be applied to this situation, as we shall explain in Section \ref{sec_operatorW}. Instead, we rely on a more pedestrian method.
We shall construct a formal power series representing the function $h_\lambda$, in the form
\begin{equation}
	\label{eq_formalpowerseries}
	h_{\lambda}(t) = t - \frac \pi 3 + \sum_{n=1}^\infty \frac{(\lambda-\lambda_0)^n}{n!} \varphi_n(t),
\end{equation}
where $\varphi_n$ are trigonometric polynomials.
Consider the partial sums
\[h_{\lambda,N}(t) = t - \frac \pi 3 + \sum_{n=1}^N \frac{(\lambda-\lambda_0)^n}{n!} \varphi_n(t).\]
The functions $\varphi_n$ are constructed in such a way that for every $N \in \N$, and any hyperplane $H$ tangent to the ball, 
the $3$-dimensional volume $f(\lambda) = \vol{H \cap K_{h_\lambda}}$ satisfies
\[
	f(\lambda_0)=\lambda_0, \quad f'(\lambda_0) = 1, \quad f''(\lambda_0)=0, \quad \ldots  \quad f^{(N)}(\lambda_0) = 0.
\]
In other words, for $\lambda$ close to $\lambda_0$, the volume of the section of $K_{h_\lambda}$ is $\lambda$ up to order $(\lambda-\lambda_0)^{N+1}$.
These equations determine uniquely the functions $\varphi_n$ for all $n$, except for $n=1$ for which we have $2$ choices; one choice coincides with the direction of the trivial branch of solutions, and the other choice will generate the non-trivial one.

Unfortunately we were not able to prove the convergence of the power series \eqref{eq_formalpowerseries}, so we have to settle for the approximate solutions $h_{\lambda, N}$.

The concrete example $K_{\lambda_0+1, 8}$ is already a very good approximation of a counterexample.
		The difference between outradius and inradius can be computed exactly, finding symbolically the maximum and minimum of $h_{\lambda_0+1,8}$ and using \eqref{eq_intro_curve}. The computations were done in Mathematica 14.2 and give
		\[R(K_{\lambda_0+1,8}) - r(K_{\lambda_0+1,8}) > 0.176.\]
		Numerical simulations show that $A_{\lambda_0+1,8}$ oscillates inside an interval of length $3 \times 10^{-7}$ near $22.76559237$.

\section{Sections of a convex body of revolution}

\subsection{Parametrization}
\label{sec_param}

Let $\psi:\R \to (0,\pi/2)$ be an even, $C^1$ and $2\pi$-periodic function with $\psi'(t)<1$ for every $t \in \R$.
Consider the curve $C$ parametrized in polar coordinates by the function
\[\rho(t) = \cos(\psi(t))^{-1},\]
where $\rho(t)$ represents the radial function at angle $t \in [0,2\pi)$.
The points of $C$ are parametrized as 
\[\gamma(t) = \rho(t) (\cos(t), \sin(t)),\]
and notice that $\|\gamma(t)\| > 1$ for every $t \in [0,2\pi)$, meaning that $\B_2$ is contained in the interior of $C$.
Since $\psi$ is even, it is clear that $C$ is a symmetric curve with respect to the horizontal axis.

\begin{figure}
	\caption{}
	\label{fig_parametrization}
\includegraphics[scale=.5]{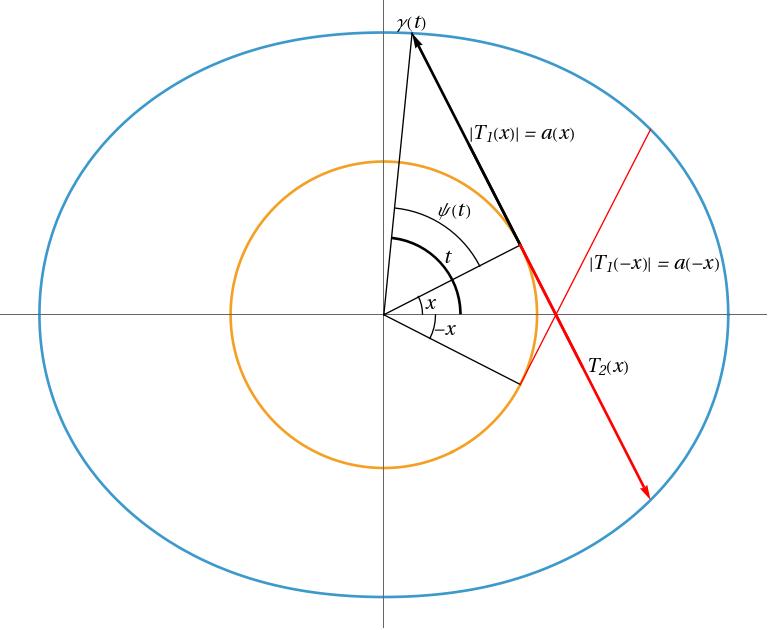}
\end{figure}

Let
\begin{equation}
	\label{eq_def_h_psi}
	h(t) = t - \psi(t),
\end{equation}
and put $x = h(t)$ as in Figure \ref{fig_parametrization}.
Since $\rho(t) \cos(\psi(t)) = 1$ and $\psi(t)$ is the angle between $\gamma(t)$ and $(\cos(x), \sin(x))$, the vector $T_1$ from $(\cos(x), \sin(x))$ to $\gamma(t)$ is tangent to $S^1$ (see Figure \ref{fig_parametrization}).
It points in the direction $(\cos(x+\pi/2), \sin(x+\pi/2))$, and its length can be computed as
\[\tan(\psi(t)) = \tan(t - h(t)).\]

By the condition imposed on $\psi'$, the function $h$ is strictly increasing and thus, invertible. It can be regarded as a diffeomorphism of the circle, since
\[h(t+2\pi) = h(t) + 2\pi,\]
for every $t \in \R$.
The invertibility of $h$ allows us to parametrize the curve with respect to the variable $x$.
Moreover, the invertibility is equivalent to the fact that every line tangent to the unit circle $S^1 = \partial \B_2$, touches $C$ in exactly two points (again, see Figure \ref{fig_parametrization}).
We regard $T_1 = T_1(x)$ as a function of $x$, and so its length,
\[a(x) = \tan(h^{-1}(x)-x).\]

Now we intend to compute the length of the other tangent vector to $S^1$, say $T_2(x)$, from $(\cos(x), \sin(x))$ to a point in $C$, and having the direction $(\cos(x-\pi/2), \sin(x-\pi/2))$.
Since we already know that $C$ is symmetric with respect to the horizontal axis, and $(\cos(-x), \sin(-x)) = (\cos(x), -\sin(x))$, 
the vector from $(\cos(-x), \sin(-x))$ to $\gamma(h^{-1}(-x))$, which is tangent to $S^1$, must reflect onto a vector from $(\cos(x), \sin(x))$ to a point in $C$, which is also tangent to $S^1$.
We deduce that $a(-x)$ is the length of $T_2(x)$.

In sum, we have proven the following:
\begin{lemma}
	Let $\psi: \R \to (0,\pi/2)$ be an even, $C^1$ and $2\pi$-periodic function such that $\psi'(x) < 1$ for every $x \in \R$.
	Consider $h(t) = t - \psi(t)$ which is a diffeomorphism of the real line, satisfying
	\[h(t+2\pi) = 2\pi + h(t),\]
	for every $t \in \R$ (and thus it represents a diffeomorphism of $S^1$).
	
	\begin{enumerate}[label=(\alph*), ref = \alph* ]
		\item
	The functions
	\begin{align}
        \label{eq_parametrization_h_plus} \delta_{K,+}(x) &= (\cos(x), \sin(x)) +\tan\left(h^{-1}(x) - x\right) (\cos(x+\pi/2), \sin(x+\pi/2)),\\
	\label{eq_parametrization_h_minus} \delta_{K,-}(x) &= (\cos(x), \sin(x)) +\tan\left(h^{-1}(-x) + x\right) (\cos(x-\pi/2), \sin(x-\pi/2)),\\
	\label{eq_parametrization_gamma} \gamma_K(t) &= \cos(t-h(t))^{-1} (\cos(t), \sin(t)),
	\end{align}
	where $x,t \in [0,2\pi)$, parametrize the same set $C \subseteq \R^2$, which is a $C^1$ closed curve, symmetric with respect to the horizontal axis.
	\item
		For every $x \in \R$, the (only) two segments $T_1, T_2$ tangent to $S^1$, joining $(\cos(x), \sin(x))$ with $p_1, p_2 \in C$ have lengths $a(x) = \tan(h^{-1}(x) - x)$ and $a(-x) = \tan(h^{-1}(-x) + x)$, respectively.

		\item
	If $\psi$ is also $\pi$-periodic, then $C$ is symmetric with respect to the origin.

		\item
			For every $x \in \R$,
	\begin{equation}
		\label{eq_relation_a_phi_h}
		h^{-1}(x) - x = \psi(h^{-1}(x)).
	\end{equation}

		\item
	If $h$ is a $C^\infty$ diffeomorphism of the real line, then $C$ is also a $C^\infty$ curve.

		\item
			If $\psi(t) = \alpha + \varphi(t)$ where $\alpha>0$ and $\varphi$ is sufficiently small in the $C^2$ topology, then $C$ is the boundary of a convex body which is close, in the Hausdorff metric, to the Euclidean ball of radius $\cos(\alpha)^{-1}$.
\end{enumerate}
\end{lemma}

\subsection{Area equation}
\label{sec_equation}

Let $K \subseteq \R^4$ be a symmetric convex body of revolution, whose axis of revolution is the vector $e_1 = (1,0,0,0)$.
The set $K$ can be described with a one dimensional concave and even function $f:[-d,d] \to \R_+$ as 
\[K = \{(t,y) \in [-d,d] \times \R^3 : \|y\|_2 \leq f(t) \}.\]
The intersection of $K$ with the plane $\R^2 \times \{(0,0)\}$ can be identified with a subset of $\R^2$, giving us the following planar convex body:
\[M = \{(t,y) \in [-d,d] \times \R : |y| \leq f(t) \}.\]
For any line $L \subseteq \R^2$ tangent to $\B_2$, we consider the hyperplane
\[
	H_L = L + \{(0,0)\} \times \R^2.
\]

Let $s \in \R$ be the slope of the line $L$. It was proven in \cite[Lemma 4.2]{alfonseca2025croft} that the $3$-dimensional volume of the $1$-section of $K$ by $H_L$, is given by
\[
	\omega_2 \sqrt{s^2+1} \int_{l(s)}^{r(s)} \left(f(\tau)^2-(\sqrt{s^2+1}+s \tau)^2\right) d\tau,
\]
where $l(s),r(s)$ are the horizontal coordinates of the intersection of the tangent line, with the boundary of $M$.
Let us make the change of variables $s = \tan(x-\pi/2)$ which is the slope of the tangent line given in Figure \ref{fig_parametrization}.
Parametrizing the boundary of $M$ as in Section \ref{sec_param}, we see that its intersection with the tangent line is given by equations \eqref{eq_parametrization_h_plus} and \eqref{eq_parametrization_h_minus}.
We get
\begin{equation}
	\label{eq_area_x}
	A_K(x) = \frac{\omega_2}{\sin(x)} \int_{\cos(x)-a(x) \sin(x)}^{\cos(x) + a(-x) \sin(x)} \left(f(\tau)^2-\left( \frac{1-\tau \cos(x)}{\sin(x)} \right)^2 \right) d\tau,
\end{equation}
where $a(x) = \tan(h^{-1}(x) - x)$. Now we compute
\begin{align}
	\frac{\partial}{\partial x} (\sin(x) A_K(x)) 
	&= 2\pi \sin(x)^{-3} \int_{\cos(x) - a(x) \sin(x)}^{\cos(x) + a(-x) \sin(x)} (\tau \cos(x)-1)(\tau-\cos(x)) d\tau \\
	&= 2\pi \sin(x)^{-3} \left.\left( \cos(x) \frac {\tau^3}3 - (\cos(x)^2+1) \frac{\tau^2}2 +\cos(x) \tau \right) \right|_{\cos(x) - a(x) \sin(x)}^{\cos(x) + a(-x) \sin(x)} \\
	\label{eq_darea_x} &= \pi \sin(x) ( a(x)^2 - a(-x)^2 ) + \frac {2\pi}3 \cos(x) (a(x)^3 + a(-x)^3),
\end{align}
where significant cancelations take place from the second line to the third one.

If $A_K(x) = A_K$ is constant, then we get the equation
\begin{equation}
	\label{eq_area_constant}
	A_K \cos(x) = \pi \sin(x) ( a(x)^2 - a(-x)^2 ) + \frac{2\pi}3 \cos(x) (a(x)^3 + a(-x)^3).
\end{equation}

\section{The non-linear operator $W$}
\label{sec_operatorW}

Let $\operatorname{Diff}_+(S^1, S^1)$ denote the set of orientation preserving diffeomorphisms of the circle.
Parametrizing the points of $S^1$ as $(\cos(x),\sin(x))$ with $x \in \R$, it can be identified naturally with the following set:
\begin{multline}
	\label{eq_def_diff}
	\operatorname{Diff}_+(S^1, S^1) = \{h: \R \to \R: \text{$h$ is of class $C^\infty$}, \\
	h(t+2\pi) = h(t) + 2\pi \text{ and } h'(t)>0\ \forall t\in\R \}.
\end{multline}

In view of equations \eqref{eq_darea_x} and \eqref{eq_area_constant}, and the discussion in Section \ref{sec_param}, we define the following operator,
\[
	W: \operatorname{Diff}_+(S^1, S^1) \to C^\infty(S^1,\R).
\]
\begin{multline}
	\label{eq_def_W}
	W(h)(x) = \pi \sin(x) \left( \tan(h^{-1}(x) - x )^2 - \tan(h^{-1}(-x)+x)^2\right) \\
	+ \frac{2\pi}3 \cos(x) \left( \tan(h^{-1}(x) - x )^3 + \tan(h^{-1}(-x)+x)^3\right) .
\end{multline}

The set of orientation-preserving diffeomorphisms of the circle $\operatorname{Diff}_+(S^1, S^1)$ is an open subset of $C^\infty(S^1, S^1)$, which is defined as in \eqref{eq_def_diff} but without the condition $h'>0$.
As in Section \ref{sec_param}, any $h \in C^\infty(S^1, S^1)$ is written as $h(t) = t + \varphi(t)$ where $\varphi$ is a $2\pi$-periodic real function, so the space $C^\infty(S^1, S^1)$ is an affine subspace (a vector subspace translated by the identity function) of $C^\infty(\R,\R)$, whose tangent vector space is $C^\infty(S^1, \R)$.
Also, in order for $h(t)$ to represent an invertible map of the circle, we need $\varphi'(t) > -1$ for every $x \in \R$. This condition guarantees that $h \in \operatorname{Diff}_+(S^1, S^1)$.
As mentioned in Section \ref{sec_param}, we restrict $\varphi$ to be an even, $\pi$-periodic function. Notice that $W(h)$ is always even and $\pi$-antiperiodic, meaning that $W(h)(x+\pi) = -W(h)(x)$ for every $h$ and $x$.

Summarizing, $W$ is defined as
\[W: e + U \subseteq e + X \to Y\]
where
\begin{align}
	X &= \{\varphi \in C^\infty(\R,\R) : \varphi(t+\pi) = \varphi(t), \varphi(-t) = \varphi(t)\},\\
	U &= \{\varphi \in X: 1+\varphi'(t) > 0\},\\
	Y &= \{\varphi \in C^\infty(\R,\R) : \varphi(x+\pi) = -\varphi(x), \varphi(-x) = \varphi(x)\},
\end{align}
and $e = \id_\R$.
Notice that the choice of variables $t,x$ for functions in $X$ and $Y$ respectively, are consistent with the angles in Section \ref{sec_param}

Using Fourier series, we can easily describe bases for $X,Y$:
\begin{align}
	\left\{ \cos(2 n t): n = 0, 1, 2 \ldots \right\} & \subseteq X,\\
	\left\{ \cos((2n+1) x): n = 0, 1, 2 \ldots \right\} & \subseteq Y.
\end{align}
In this paper we will work only with trigonometric polynomials and it will not be necessary to deal with any topological considerations for the spaces $X,Y$.
In other words, we will treat these bases as algebraic bases.

\subsection{Approximation and differentials of operators}
Now we start the approximation process.
The starting point is the function
\[h_{\lambda_0}(t) = t-\pi/3,\]
which represents the circumference of radius $2$.
The application of $W$ gives
\[W(h_{\lambda_0})(x) = \lambda_0 \cos(x).\]
Here the value of $\lambda_0 = \omega_3 3^{3/2} = 4 \sqrt{3}\pi$ is the $3$-dimensional volume of the $1$-sections of the ball of radius $2$ in $\R^4$.

We search a solution of the equation
\begin{equation}
	W(h_\lambda)(x) = \lambda \cos(x),
\end{equation}
of the form
\[h_\lambda(t) = t+ \sum_{n=0}^\infty \frac {(\lambda-\lambda_0)^n}{n!} \varphi_n(t),\]
where $\varphi_n$ are trigonometric polynomials. Clearly $\varphi_0(t) = -\pi/3$.
Notice here that, in the notation of Section \ref{sec_param},
\begin{align}
\psi(t)
&= - \sum_{n=0}^\infty \frac {(\lambda-\lambda_0)^n}{n!} \varphi_n(t) \\
&= \frac \pi 3 - \sum_{n=1}^\infty \frac {(\lambda-\lambda_0)^n}{n!} \varphi_n(t).
\end{align}

The application of $W$ can be written in the form
\[W(h_\lambda)(x) = \lambda_0 \cos(x) + \sum_{n=1}^\infty \frac {(\lambda-\lambda_0)^n}{n!} \kappa_n(x),\]
since one can see that $W(h_\lambda)(x)$ is an analytic function of the variables $\lambda, x$.
This will be proved in Proposition \ref{prop_analytic} in Section \ref{sec_error_estimates}.

Each $\kappa_n$ is an analytic function of one variable, and they turn out to be trigonometric polynomials, as we shall see in this section.
To compute the functions $\kappa_n$ we need the language of higher differentials and multilinear operators in vector spaces.

We define the $n$-th differential of a function $F: e+U \to Y$ as the $n$-linear operator $F^{(n)}(f_0): X \times \cdots \times X \to Y$ given by
\[
	F^{(n)}(f_0)[v_1, \ldots, v_n](x) = \left. \frac{\partial^n}{\partial t_1 \cdots \partial t_n} F(f_0 + t_1 v_1 + \cdots + t_n v_n)(x) \right|_{t_1 = \cdots = t_n=0}.
\]
This is a symmetric multilinear operator in the variables $v_1, \ldots, v_n \in X$.
The Faa-di Bruno formula computes the $n$-th differential of a composition as
\begin{equation}
	\label{eq_faadibruno}
	\left. \left( \frac\partial{\partial \lambda}\right)^N F(h_\lambda)(x) \right|_{\lambda=0} = 
		\sum_{m_1 + 2 m_2 + \cdots + N m_N = N} C_{N, m_i} F^{(r)}[\varphi_1^{[m_1]}, \ldots \varphi_N^{[m_N]}](x),
\end{equation}
where the sum is over the non-negative integers $m_1, \ldots, m_N$, and where
\[C_{N, m_i} = \frac {N!}{m_1! \cdots m_N! 1!^{m_1} 2!^{m_2} \cdots N!^{m_N}}.\]

For example,
\begin{align}
	\left. \frac\partial{\partial \lambda} F(h_\lambda)(x) \right|_{\lambda=0} &= F'(f_0)[\varphi_1](x), \\
		\left. \left(\frac\partial{\partial \lambda}\right)^2 F(h_\lambda)(x) \right|_{\lambda=0} &= F''(f_0)[\varphi_1, \varphi_1](x) + F'(f_0)[\varphi_2](x), \\
			\left. \left(\frac\partial{\partial \lambda}\right)^3 F(h_\lambda)(x) \right|_{\lambda=0} &= F'''(f_0)[\varphi_1, \varphi_1, \varphi_1](x) + 3F''(f_0)[\varphi_1, \varphi_2](x) + F'(f_0)[\varphi_3](x).
\end{align}

This way we may compute 
\begin{align}
	\kappa_1 &= W'(f_0)[\varphi_1] \label{eq_faa_1} \\
	\kappa_2 &= W''(f_0)[\varphi_1, \varphi_1] + W'(f_0)[\varphi_2] \label{eq_faa_2} \\
	\kappa_3 &= W'''(f_0)[\varphi_1, \varphi_1, \varphi_1] + 3W''(f_0)[\varphi_1, \varphi_2] + W'(f_0)[\varphi_3] \\
	\cdots \\
	\kappa_N &= W^{(n)}(f_0)[\varphi_1, \ldots, \varphi_1] + \cdots + N W''(f_0)[\varphi_1, \varphi_{N-1}] + W'(f_0)[\varphi_N]
\end{align}
and the goal is to find a sequence $\varphi_n$ such that $\kappa_n \equiv 0$ for every $n \geq 1$, and $\kappa_1(x) = \cos(x)$.

It is of vital importance that $\kappa_N(x)$ has the form
\[A(x) + B[\varphi_1, \varphi_{N-1}](x) + L[\varphi_N](x),\]
where $A$ is independent of $\varphi_{N-1}, \varphi_N$, $B$ is bilinear and $L$ is linear.
Moreover, the operators $B[\varphi_1, \cdot]$ and $L$ will be computed explicitly.

\subsection{Differential, Kernel and Image}

The differential of $W$ can be computed as
\[W'(h)[V] := \left.\frac\partial {\partial t} \right|_{t=0}W(h+t V).\]

After lengthy computations which we omit, we arrive at the following formula:
\[W'(f_0)[V](x) = -16 \sqrt{3} \pi V\left(x-\frac{\pi }{3}\right) \cos \left(x+\frac{\pi }{6}\right)-16 \sqrt{3} \pi V\left(x+\frac{\pi }{3}\right) \cos \left(\frac{\pi }{6}-x\right)\]

The following table describes the values of $-\frac 1{24 \pi} W'(h)$, in the specified basis of $X$ and $Y$.
The $(n+1)$-th column corresponds to the coefficients of $-\frac 1{24\pi} W'(f_0)[\cos(2n \ \cdot\ )]$ in the basis $\{\cos(x), \cos(3x), \ldots\}$.
\begin{equation}
	\label{tab_D1}
\begin{array}{c|ccccccccc}
 \text{} & 1 & \cos (2 t) & \cos (4 t) & \cos (6 t) & \cos (8 t) & \cos (10 t) & \cos (12 t) & \cos (14 t) \\
 \hline
 \cos (x) & 2 & -1 & 0 & 0 & 0 & 0 & 0 & 0 \\
 \cos (3 x) & 0 & 0 & 0 & 0 & 0 & 0 & 0 & 0 \\
 \cos (5 x) & 0 & 0 & -1 & 1 & 0 & 0 & 0 & 0 \\
 \cos (7 x) & 0 & 0 & 0 & 1 & -1 & 0 & 0 & 0 \\
 \cos (9 x) & 0 & 0 & 0 & 0 & 0 & 0 & 0 & 0 \\
 \cos (11 x) & 0 & 0 & 0 & 0 & 0 & -1 & 1 & 0 \\
 \cos (13 x) & 0 & 0 & 0 & 0 & 0 & 0 & 1 & -1 \\
 \cos (15 x) & 0 & 0 & 0 & 0 & 0 & 0 & 0 & 0 \\
\end{array}
\end{equation}
We see that the kernel and image are given by
\begin{align}
	\ker(W'(h_{\lambda_0})) &= \left\langle \cos ((6 n-2) t) + \cos (6 n t) + \cos ((6 n+2) t): n = 0,1,2, \ldots \right\rangle,\\
	\im(W'(h_{\lambda_0})) &= \left\langle \cos ((6 n+1)x), \cos((6n+5)x) : n = 0,1,2,\ldots \right\rangle.
\end{align}
The first generator of $\ker(W'(h_{\lambda_0}))$ is with $n=0$, that is, $1 + 2 \cos(2t)$.
We denote
\[
	K_n(t) = \cos ((6 n-2) t) + \cos (6 n t) + \cos ((6 n+2) t)
\]
for $n\geq 0$.

\section{Approximation}
\label{sec_approximation}

\subsection{First derivative}

To compute the first term $\varphi_1$ we take the first derivative in
\[W(\id_\R + \varphi_0 + (\lambda-\lambda_0) \varphi_1 )(x) = \lambda \cos(x)\]
at $\lambda = \lambda_0$,
and obtain
\[
	W'(h_{\lambda_0})[\varphi_1](x) = \cos(x).
\]
The solution will have the form $\varphi_1 = w_1 + k_1$ where $w_1$ is a particular solution, and $k_1 \in \ker(W'(h_{\lambda_0}))$.
In view of the first column of the table for the first differential, we have a particular solution
\[
	w_1(t) = -\frac 1{48\pi}.
\]
To determine $k_1$ we need to consider the second differential.

\subsection{Second derivative, the bifurcation direction}

The solution of $\varphi_1$ will be of the form $\varphi_1(t) = w_1(t) + \mu_0 K_0(t) + \mu_1 K_1(t) + \cdots$ with $\mu_i \in \R$.
We propose a solution in a simple form
\[\varphi_1(t) = w_1(t) + \mu_0 K_0(t) = -\frac 1{48 \pi} + \mu_0 (1+2\cos(2t)),\]
for some $\mu_0 \in \R$.

Take the second derivative of
\[W\left(\id_\R + \varphi_0 + (\lambda-\lambda_0) \varphi_1 + \frac{(\lambda-\lambda_0)^2}2 \varphi_2 + o((\lambda-\lambda_0)^2) \right)(x) = \lambda \cos(x),\]
with respect to $\lambda$ at $\lambda = \lambda_0$.
By \eqref{eq_faa_2} we get
\begin{align}
	\label{eq_step2}
	W''(h_{\lambda_0})[w_1 + \mu_0 K_0, w_1 + \mu_0 K_0] + W'(h_{\lambda_0})[\varphi_2] = 0.
\end{align}
In order for this equation to have a solution it is necessary that 
\begin{equation}
	\label{eq_image_condition}
	W''(h_{\lambda_0})[w_1 + \mu_0 K_0, w_1 + \mu_0 K_0] \in \im(W'(h_{\lambda_0})).
\end{equation}
The left hand side (we omit again the computations) in the basis of $Y$ is
\[
	\frac{7 \left(2304 \pi ^2 \mu _0^2+1\right)}{24 \sqrt{3} \pi } \cos (x)
	+4 \sqrt{3} \left(1-24 \pi \mu _0\right) \mu _0 \cos (3 x)
	-128 \sqrt{3} \pi \mu _0^2 \cos (5 x).
\]
For this function to belong to $\im(W'(h_{\lambda_0}))$ we need the second term to vanish because the image is generated by $\cos(x), \cos(5x), \ldots$ (see the matrix \eqref{tab_D1}), then
\[(1-24 \pi \mu _0)\mu _0=0.\]
This is a quadratic equation in $\mu_0$. The two solutions are $\mu_0=0, \frac 1{24 \pi}$.
The first solution $\mu_0 = 0$ will generate the trivial branch of solutions, the second one $\mu_0 = \frac 1{24 \pi}$ will generate the bifurcation.
To see this, recall that
\[W[\id_\R - \arctan((\lambda/\omega_3)^{1/3}) ](x) = \lambda \cos(x),\]
so the trivial branch of solutions is given by
\[t + \sum \frac{(\lambda-\lambda_0)^n}{n!} \varphi_n^{(\text{trivial})}(t) = t - \arctan((\lambda/\omega_3)^{1/3}),\]
and for every $n$, $\varphi_n^{(\text{trivial})}(t)$ is the constant function,
\[\varphi_n^{(\text{trivial})}(t) = \left. - \frac{d^n}{d\lambda^n}\arctan((\lambda/\omega_3)^{1/3}) \right|_{\lambda=\lambda_0}.\]
The first values are
\[
	\varphi_0^{(\text{trivial})}(t) = -\frac{\pi}3,
	\quad \varphi_1^{(\text{trivial})}(t) = -\frac 1{48 \pi} = w_1(t),
	\quad \varphi_2^{(\text{trivial})}(t) =\frac {7}{1152 \sqrt 3 \pi^2}.
\]

Going back to the non-trivial solution, we take $\mu_0 = \frac 1{24 \pi}$ and we get
\begin{equation}
	\label{eq_varphi1}
	\varphi_1(t) = -\frac 1{48 \pi} + \frac 1{24 \pi} (1+2\cos(2t)).
\end{equation}

One could also propose a more general solution
\[\varphi_1 = w_1 + \mu_0 K_0 + \mu_1 K_1 + \cdots,\]
but the condition \eqref{eq_image_condition} yields always $\mu_0 = 0$ or $\frac 1{24 \pi}$, and $\mu_i = 0$ for $i \geq 1$.

To find $\varphi_2$ we replace the computed value for $\mu_0$ in \eqref{eq_step2}, and compute the second differential to get
\begin{equation}
	\label{eq_step2w}
	W'(h_{\lambda_0})[w_2](x) = \frac{35}{24 \sqrt{3} \pi } \cos (x) - \frac{2}{3 \sqrt{3} \pi } \cos (5 x).
\end{equation}

To solve this linear equation we use part of the matrix in Table \ref{tab_D1},
\[
\left(
\begin{array}{cc}
 24 \pi & 0 \\
 0 & 0 \\
 0 & 24 \pi \\
\end{array}
\right).\left(
\begin{array}{c}
 a_1 \\
 a_2 \\
\end{array}
\right)=\left(
\begin{array}{c}
	-\frac{35}{24 \sqrt{3} \pi } \\
	 0 \\
	\frac{2}{3 \sqrt{3} \pi }
\end{array}
\right),
\]
and its solution is
\begin{align}
	w_2(t) 
	&= a_1 \cos (2 t) + a_2\cos (4 t) \\
	&= -\frac{35}{576 \sqrt{3} \pi ^2} \cos (2 t) + \frac{1}{36 \sqrt{3} \pi ^2}\cos (4 t)
\end{align}
Again, this is a particular solution of \eqref{eq_step2w}.
The general solution for $\varphi_2$ is
\[\varphi_2 = w_2 + k_2\]
where $k_2 \in \ker(W'(h_{\lambda_0}))$.

\subsection{Third derivative and beyond}

For higher-order terms, the scheme will be always the same.
Assume inductively that we have computed trigonometric polynomials $\varphi_n \in X$ for $n = 0, \ldots, N-1$, and $w_N \in X$, which satisfy that
\[
	\left.\left(\frac\partial{\partial \lambda} \right)^N W\left( t + \sum_{n=0}^{N-1} \frac {(\lambda-\lambda_0)^n}{n!} \varphi_n(t) + \frac {(\lambda-\lambda_0)^N}{N!} w_N(t) \right) \right|_{\lambda=\lambda_0} = 0.
\]
We want to compute $k_N \in \ker(W'(h_{\lambda_0})), w_{N+1} \in X$ such that
\begin{multline}
	\left. \left(\frac\partial{\partial \lambda} \right)^{N+1} W\left( t +\sum_{n=0}^{N-1} \frac {(\lambda-\lambda_0)^n}{n!} \varphi_n(t) \right.\right. \\ \left. \left. + \frac {(\lambda-\lambda_0)^N}{N!} (w_N(t) + k_N(t)) + \frac {(\lambda-\lambda_0)^{N+1}}{(N+1)!} w_{N+1}(t) \right) \right|_{\lambda=\lambda_0} = 0.
\end{multline}
Then we define $\varphi_N = w_N + k_N$, and use $\varphi_N, w_{N+1}$ for the next step.

According to \eqref{eq_faadibruno}, the $(N+1)$-th derivative has the form
\[ A(x) + (N+1) W''(h_{\lambda_0})[\varphi_1, w_N + k_N](x) + W'(h_{\lambda_0})[w_{N+1}](x) = 0,\]
where $A$ collects all the terms involving higher differentials of $W$.
We need to solve for $(k_N, w_{N+1})$ the equation
\begin{equation}
	\label{eq_stepN}
	(N+1) W''(h_{\lambda_0})[\varphi_1, k_N] + W'(h_{\lambda_0})[w_{N+1}] = - A(x) - (N+1) W''(h_{\lambda_0})[\varphi_1, w_N],
\end{equation}
which is a linear equation. Let us show that this equation has always a solution.

The first term in the left-hand side can be computed explicitly: after lengthy computations which we omit, we arrive at

\begin{multline}
	\label{eq_DW2}
	W''(h_{\lambda_0})[\varphi_1, K_N] =
	\frac 1{\sqrt 3} \left(
	(-6 N-8) \cos ((6 N-5) x)+(6 N-3) \cos ((6 N-3) x) \right. \\ \left. 
		+(14-6 N) \cos ((6 N-1) x)+(6 N+14) \cos ((6 N+1) x) \right. \\ \left. 
			+(-6 N-3) \cos ((6 N+3) x)+(6 N-8) \cos ((6 N+5) x)
\right).
\end{multline}
The following table shows the coordinates of $\sqrt 3 W''(h_{\lambda_0})[\varphi_1, K_n]$ for $n = 0, \ldots 5$, in the basis of $Y$.
\begin{equation}
	      \label{tab_D2}
\begin{array}{c|cccccc}
 \text{} & K_0 & K_1 & K_2 & K_3 & K_4 & K_5 \\
 \hline
 \cos (x) & 28 & -14 & 0 & 0 & 0 & 0 \\
 \cos (3 x) & -6 & 3 & 0 & 0 & 0 & 0 \\
 \cos (5 x) & -16 & 8 & 0 & 0 & 0 & 0 \\
 \cos (7 x) & 0 & 20 & -20 & 0 & 0 & 0 \\
 \cos (9 x) & 0 & -9 & 9 & 0 & 0 & 0 \\
 \cos (11 x) & 0 & -2 & 2 & 0 & 0 & 0 \\
 \cos (13 x) & 0 & 0 & 26 & -26 & 0 & 0 \\
 \cos (15 x) & 0 & 0 & -15 & 15 & 0 & 0 \\
 \cos (17 x) & 0 & 0 & 4 & -4 & 0 & 0 \\
 \cos (19 x) & 0 & 0 & 0 & 32 & -32 & 0 \\
 \cos (21 x) & 0 & 0 & 0 & -21 & 21 & 0 \\
 \cos (23 x) & 0 & 0 & 0 & 10 & -10 & 0 \\
 \cos (25 x) & 0 & 0 & 0 & 0 & 38 & -38 \\
 \cos (27 x) & 0 & 0 & 0 & 0 & -27 & 27 \\
 \cos (29 x) & 0 & 0 & 0 & 0 & 16 & -16 \\
\end{array}
\end{equation}

One can see that for every $n$ there is a $3n \times 3n$ invertible matrix $M_{3n}$ formed by the first $n$ columns of the matrix \eqref{tab_D2}, and the columns $2, 3, 5,6, 8,9, \ldots 3n$ of the matrix \eqref{tab_D1}.
For example, for $n=2$ we have
\begin{equation}
	M_6 = \label{tab_D12}
	\left(
\begin{array}{cccccc}
 24 \pi & 0 & 0 & 0 & \frac{28}{\sqrt{3}} & -\frac{14}{\sqrt{3}} \\
 0 & 0 & 0 & 0 & -2 \sqrt{3} & \sqrt{3} \\
 0 & 24 \pi & 0 & 0 & -\frac{16}{\sqrt{3}} & \frac{8}{\sqrt{3}} \\
 0 & 0 & 24 \pi & 0 & 0 & \frac{20}{\sqrt{3}} \\
 0 & 0 & 0 & 0 & 0 & -3 \sqrt{3} \\
 0 & 0 & 0 & 24 \pi & 0 & -\frac{2}{\sqrt{3}} \\
\end{array}
\right).
\end{equation}
Now recall that $A(x)$ is a trigonometric polynomial, as it is a sum of products of $\varphi_n, n=0, \ldots, N$ and their derivatives.
To find a solution of equation \eqref{eq_stepN}, one just takes the matrix $M_{3n}$ with $3n$ larger than the degree of $A(x)$ as a trigonometric polynomial.
Notice that the coefficient $(-6N-3)$ of $\cos((6N+3) x)$ is always non-zero. This is enough to triangulate $M_{3n}$ by lines, obtaining the identity matrix. 

Since $M_{3n}$ is invertible for every $n$, we just proved that equation \eqref{eq_stepN} always has a solution.

\section{Error estimates}
\label{sec_error_estimates}
At this point we know the existence of trigonometric polynomials $\varphi_n$ for all $n \geq 0$, with the property that
\[
	\left. \left( \frac\partial{\partial \lambda}\right)^N W\left(\id_\R + \sum_{n=0}^N \frac {(\lambda-\lambda_0)^n}{n!} \varphi_n\right)(x) \right|_{\lambda=\lambda_0} = 0.
\]

We show here the first terms of the series, which were computed symbolically using Mathematica 14.2:
\begin{align}
	\varphi _1(x) &= \frac{1}{48 \pi } + \frac{\cos (2 x)}{12 \pi } \\
 \varphi _2(x) &= \frac{\sqrt{3}}{128 \pi ^2} -\frac{\cos (2 x)}{72 \sqrt{3} \pi ^2}+\frac{\cos (4 x)}{36 \sqrt{3} \pi ^2} \\
 \varphi _3(x) &= -\frac{41}{10368 \pi ^3} + \frac{35 \cos (2 x)}{1296 \pi ^3}+\frac{49 \cos (6 x)}{10368 \pi ^3}-\frac{49 \cos (4 x)}{10368 \pi ^3}-\frac{67 \cos (8 x)}{10368 \pi ^3}\\
 \varphi _4(x) &= \frac{8513}{331776 \sqrt{3} \pi ^4} -\frac{479 \cos (2 x)}{10368 \sqrt{3} \pi ^4}+\frac{293 \cos (4 x)}{6912 \sqrt{3} \pi ^4}+\frac{755 \cos (8 x)}{31104 \sqrt{3} \pi ^4}-\frac{101 \cos (6 x)}{10368 \sqrt{3} \pi ^4}\\& -\frac{67 \cos (10 x)}{31104 \sqrt{3} \pi ^4} \\
 \varphi _5(x) &= -\frac{126461}{5971968 \pi ^5} + \frac{125035 \cos (2 x)}{1492992 \pi ^5}+\frac{138595 \cos (6 x)}{4478976 \pi ^5}+\frac{1757 \cos (10 x)}{4478976 \pi ^5}\\& +\frac{22445 \cos (14 x)}{4478976 \pi ^5} -\frac{86915 \cos (4 x)}{2239488 \pi ^5}-\frac{159805 \cos (8 x)}{4478976 \pi ^5}-\frac{22445 \cos (12 x)}{4478976 \pi ^5} \\
 \varphi _6(x) &=\frac{845965}{5971968 \sqrt{3} \pi ^6} -\frac{595565 \cos (2 x)}{1679616 \sqrt{3} \pi ^6}+\frac{14909705 \cos (4 x)}{53747712 \sqrt{3} \pi ^6}+\frac{5614835 \cos (8 x)}{26873856 \sqrt{3} \pi ^6}\\& +\frac{704035 \cos (12 x)}{17915904 \sqrt{3} \pi ^6} -\frac{582145 \cos (14 x)}{13436928 \sqrt{3} \pi ^6}-\frac{2444945 \cos (6 x)}{17915904 \sqrt{3} \pi ^6}-\frac{530275 \cos (10 x)}{26873856 \sqrt{3} \pi ^6}\\& -\frac{157115 \cos (16 x)}{26873856 \sqrt{3} \pi ^6} \\
\end{align}

\begin{figure}
	\caption{The curves obtained parametrized by $h_{6,\lambda}$ for the values $\lambda_0-5, \lambda_0-2, \lambda_0 + 2, \lambda_0 + 5$ from left to right.}
	\label{fig_curves}
	\includegraphics[width=\textwidth]{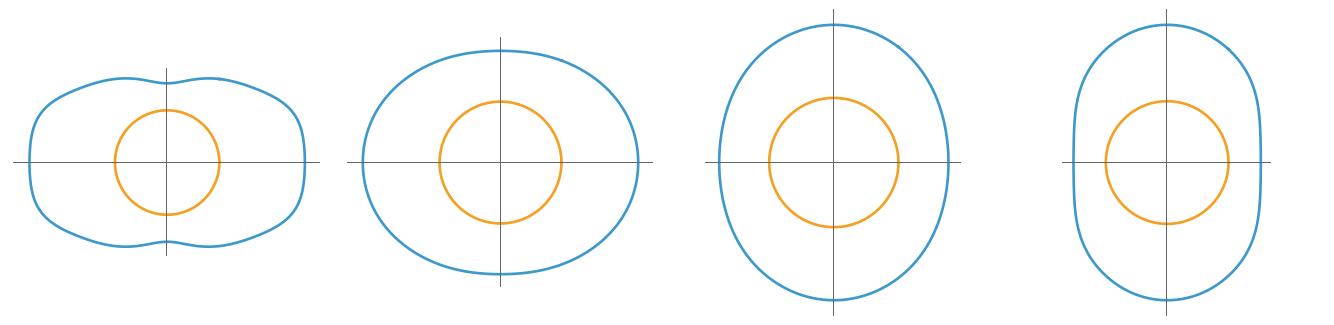}
\end{figure}

Unfortunately, we were not able to show that the power series is convergent.
Instead we will study the partial sums, as approximate counterexamples for the Barker--Larman problem.

From now on, we write
\[h_{\lambda,N}(t) = t - \frac \pi 3 + \sum_{n=1}^N \frac{(\lambda - \lambda_0)^n}{n!} \varphi_n(t)\]
for every partial sum.
Figure \ref{fig_curves} shows some examples of bodies created with any of the parametrizations \eqref{eq_parametrization_h_plus} or \eqref{eq_parametrization_h_minus}, using $h = h_{\lambda,N}$.
Notice that one can obtain also non-convex examples.

For $r>0$ and $z \in \C$ we denote $B^\C_r(z) \subseteq \C$ the open Euclidean ball of centre $z$ and radius $r$.
If $z \in \R$, we denote $B^\R_r(z) = (z-r,z+r)$ the real ball.
The closed balls will be denoted by $\bar B^\C_r(z)$ and $\bar B^\R_r(z)$.

For a function $\varphi:\R \to \R$ which is periodic, 
\[\| \varphi \|_\infty = \max_{t \in \R} |\varphi(t)| \]
denotes the $\infty$-norm in the real line. In practice these will be periodic functions, so it suffices to consider $x \in [0,\pi]$. 
\begin{proposition}
	\label{prop_analytic}
Let 
\[h_{\lambda,N}(t) = t - \frac \pi 3 + \sum_{n=1}^N \frac{(\lambda-\lambda_0)^n}{n!} \varphi_n(t)\]
and take $r>0$ such that
\begin{equation}
	\label{eq_series_bound}
	\sum_{n=1}^N \frac{r^n}{n!} \|\varphi_n'\|_\infty \leq \frac 12 \text{ and } \sum_{n=1}^N \frac{r^n}{n!} \|\varphi_n\|_\infty \leq \frac \pi {12}.
\end{equation}
Then $W(h_{\lambda,N})(x)$ is an analytic function in the variables $\lambda, x$, with $\lambda \in B^\C_r(\lambda_0), |\im(x)| < \eta$, for some $\eta>0$.
Moreover, there is a universal constant $c$ for which $|W(h_{\lambda,N})(x)| \leq c$ for every $x \in \R$ and complex $\lambda \in B^\C_r(\lambda_0)$.
\end{proposition}
\begin{proof}
	First notice that for $x \in \R$ and $\lambda \in \bar B^\C_r(\lambda_0)$ we have
	\[|h_{\lambda,N}'(t)| \geq 1/2\]
	for every $x \in [0, \pi]$.
	By the inverse function theorem, for every $x \in \R$ there is $\varepsilon_x>0$ such that
	$h_{\lambda,N}$ is injective restricted to $B^\C_{\varepsilon_x}(x)$.
	A simple compactness argument shows that there exists a uniform $\varepsilon>0$ such that
	$h_{\lambda,N}$ is injective in $|\im(z)| \leq \varepsilon$, for every $\lambda \in \bar B^\C_r(\lambda_0)$.
	Moreover, since $h_{\lambda,N}(\R) = \R$, the image $h_{\lambda,N}(\{z \in \C : |\im(z)| < \varepsilon\})$ contains a strip
	\[
		S_\eta = \{z \in \C : |\im(z)| < \eta\},
	\]
	for some fixed $\eta>0$ and every $\lambda \in B_r(\lambda_0)$.
	Considering the function
	\[
		(\lambda, z) \in \C^2 \mapsto (\lambda, h_{\lambda,N}(z)) \in \C^2.
	\]
	We may apply the inverse function theorem for holomorphic functions in several variables (see \cite[Theorem 5.5]{laurent2010holomorphic}) to deduce that its inverse
	\[
		(\lambda, z) \in B_r(\lambda_0) \times S_\eta \subseteq \C^2 \mapsto (\lambda, h_{\lambda,N}^{-1}(z)) \in \C^2
	\]
	is holomorphic. In particular, $h_{\lambda,N}^{-1}(z)$ is a holomorphic function of the two variables $(\lambda, z)$.

	Now we must show that $W(h_{\lambda, N})(x)$ is holomorphic in $B_r(\lambda_0) \times S_\eta$.
	In view of \eqref{eq_def_W}, all we need to show is that $h_{\lambda,N}^{-1}(x) -x$ stays away from the singularities of the function $\tan(z)$, for every $x \in S_\eta$.
	First notice that $|h_{\lambda,N}(t) - t + \frac \pi 3| \leq \frac \pi {12}$ for every $t \in \R, \lambda \in B_r(\lambda_0)$.
	Put $t = h_{\lambda,N}^{-1}(x) \in \R$ to obtain
	\[
		|h_{\lambda,N}^{-1}(x) - x - \frac \pi 3| \leq \frac \pi {12},
	\]
	for every $x \in \R, \lambda \in B_r(\lambda_0)$.
	By the periodicity of $h_{\lambda,N}^{-1}(x) - x$, we may shrink $\eta$ if necessary, to assume that
	\[
		|h_{\lambda,N}^{-1}(x) - x - \frac \pi 3| \leq \frac \pi {9},
	\]
	for every $x \in S_\eta, \lambda \in B_r(\lambda_0)$.

	Notice that the function $\tan(z)$ is analytic in the closed ball $\bar B^\C_{\pi/9}(-\pi/3) \subseteq \C$.
	By the compactness of the closed ball, 
	\[|W(h_{\lambda,N})(x)| \leq \max_{z \in \bar B^\C_{\pi/9}(-\pi/3)} \frac{2\pi}3 |\tan(z)|^3+2\pi |\tan(z)|^2 \leq c.\]
	where $c$ is a universal constant.
\end{proof}

\begin{proposition}
	\label{prop_bound_Rr}
	For every $N>0$ there is $r_N$ such that for $\lambda \in (\lambda_0-r_N, \lambda_0+r_N)$,
	\[ |R(K_{\lambda,N}) - r(K_{\lambda,N})| \geq c |\lambda - \lambda_0|, \]
	where $c>0$ is a universal constant.
\end{proposition}
\begin{proof}
	Let $r_N>0$ be small enough so that \eqref{eq_big_angle} and \eqref{eq_small_angle} hold for all $\lambda \in (\lambda_0-r_N,\lambda_0+ r_N)$, and such that \eqref{eq_series_bound} holds with $r = r_N$.
	According to our parametrization \eqref{eq_parametrization_gamma} of $K_{\lambda,N}$, for each $t \in [0,\pi]$ and $\lambda \in (\lambda_0-r_N, \lambda_0 + r_N)$, there is a point $\gamma_{K_{\lambda,N}}(t)$ at distance $\cos(h(t)-t)^{-1}$ from the origin.
	Consider only $\lambda \in [\lambda_0, \lambda_0 + r_N)$. Taking a smaller $r_N$ if necessary, we may bound $h_{\lambda,N}$ using only the linear part,
	\begin{align}
		\label{eq_big_angle}
		-h_{\lambda,N}(0) &= \frac \pi 3 - (\lambda-\lambda_0) \varphi_1(0) - \sum_{n=2}^N \frac {(\lambda-\lambda_0)^n}{n!} \varphi_n(0)\\
		&\leq \frac \pi 3 - \frac 12 (\lambda-\lambda_0) \varphi_1(0)\\
		&= \frac \pi 3 - \frac 5{96 \pi} (\lambda-\lambda_0),
	\end{align}
	and
	\begin{align}
		\label{eq_small_angle}
	\frac \pi 2 - h_{\lambda,N}\left(\frac \pi 2\right)
		&= \frac \pi 3 - (\lambda-\lambda_0) \varphi_1(\pi/2) - \sum_{n=2}^N \frac {(\lambda-\lambda_0)^n}{n!} \varphi_n(\pi/2)\\
		&\geq \frac \pi 3 - \frac 12 (\lambda-\lambda_0) \varphi_1(\pi/2)\\
		&= \frac \pi 3 + (\lambda-\lambda_0) \frac 1{32 \pi},
	\end{align}
	where we used \eqref{eq_varphi1} in both inequalities, and the fact that $\varphi_1(\pi/2) < 0 < \varphi_1(0)$.

	Taking $t = 0, \frac \pi 2$ in \eqref{eq_parametrization_gamma} we obtain for $\lambda \in [\lambda_0, \lambda_0 + r_N),$
	\[R(K_{\lambda,N}) \geq \cos\left(\frac \pi 2 - h_{\lambda,N}\left(\frac \pi 2\right)\right)^{-1} \geq \cos(\frac \pi 3 + \frac 1{32 \pi} (\lambda-\lambda_0))^{-1} \]
	and
	\[r(K_{\lambda,N}) \leq \cos(0 - h_{\lambda,N}(0)) \leq \cos(\frac \pi 3 - \frac 5{96 \pi} (\lambda-\lambda_0))^{-1}.\]
	Estimating the difference $R(K_{\lambda,N}) - r(K_{\lambda,N})$ by its first Taylor expansion, we obtain the result for sufficiently small  $r_N>0$.
	A similar estimate holds for $\lambda \in (\lambda_0 - r_N,\lambda_0]$.
\end{proof}

\begin{proposition}
	\label{prop_bound_error}
	For any $N \in \N$ there is $r_N > 0$ such that
	\[|W(h_{\lambda,N})(x) - \lambda \cos(x)| \leq c \left| \frac{\lambda - \lambda_0}{r_N} \right|^{N+1}\]
	for $\lambda \in B^\C_{r_N}({\lambda_0})$, where $c>0$ is a universal constant.
\end{proposition}
\begin{proof}
	Let $r_N>0$ be the value obtained in Proposition \ref{prop_analytic}.
	Fix $x \in \R$ and consider the function $f(\lambda) = W(h_{\lambda,N})(x) - \lambda \cos(x)$ which is analytic in $B^\C_{r_N}({\lambda_0})$.

	Since by construction we have $f^{(n)}(\lambda_0)=0$ for $n = 0,1, \ldots, N$, the function $f(z)/(z-\lambda_0)^{N+1}$ is also analytic.
	By the maximum modulus principle,
	\[|f(z)/(z-\lambda_0)^{N+1}| \leq \max_{\partial B^\C_{r_N}({\lambda_0})} \left|\frac{f(z)}{{r_N}^{N+1}}\right|,\]
	which yields
	\[|f(z)| \leq \left(c+(\lambda_0+r_N) \max_{z \in \partial B^\C_{r_N}(\lambda_0)} |\cos(z)| \right) \left|\frac{z-\lambda_0}{r_N}\right|^{N+1},\]
where $c>0$ is the universal constant given in Proposition \ref{prop_analytic}.
\end{proof}

\begin{proof}[Proof of Theorem \ref{thm_approximate}]
We denote $K_{\lambda,N}$ the body obtained as in Section \ref{sec_param}, using the function $h = h_{\lambda,N}$.
As well, we denote by $L_{\lambda, N}, A_{K_{\lambda, N}}$, as defined in Sections \ref{sec_intro} and \ref{sec_param} .
From \eqref{eq_darea_x} and \eqref{eq_def_W},
\begin{equation}
	\label{eq_bound_WD}
	W(h_{\lambda,N})(x) - \lambda \cos(x) = \partial_x[ (A_{K_{\lambda,N}}(x) - \lambda) \sin(x)],
\end{equation}
for every $x \in [0,\pi]$.
Let
\begin{equation}
	\label{eq_def_epsilon}
	\varepsilon = \|W(h_{\lambda,N}) - \lambda \cos\|_\infty .
\end{equation}
For $x \in [0,\pi/2]$,
\begin{align}
	|( \lambda - A_{K_{\lambda,N}}(x)) \sin(x)|
&= \left| \int_0^x \partial_x[ (\lambda - A_{K_{\lambda,N}}(x)) \sin(x) ] dx \right| \\
&\leq \pi \varepsilon x.
\end{align}

Then we obtain
\begin{equation}
	\label{eq_bound_lambdaA}
	|\lambda - A_{K_{\lambda,N}}(x)| \leq \frac{x \pi\varepsilon}{\sin(x)} \leq \frac{\pi^2\varepsilon}{2}.
\end{equation}
Furthermore, by the symmetry of $L_{\lambda,N} \subseteq \R^2$ with respect to the vertical axis, we get \eqref{eq_bound_lambdaA} for every $x \in (0,\pi)$.

Let $r_N > 0$ be sufficiently small so that for $\lambda \in (\lambda_0 - r_N, \lambda_0 + r_N)$, the set $K_{\lambda,N}$ is a smooth convex body.
By Proposition \ref{prop_bound_error} we obtain an even smaller $r_N>0$ such that for $\lambda \in (\lambda_0-r_N, \lambda_0 + r_N)$,
\begin{equation}
	\label{eq_bound_D}
	|\partial_x[ (A_{K_{\lambda,N}}(x) -\lambda) \sin(x)| \leq c\left|\frac{\lambda - \lambda_0}{r_N} \right|^{N+1}.
\end{equation}

Putting together \eqref{eq_bound_WD}, \eqref{eq_def_epsilon}, \eqref{eq_bound_lambdaA} and \eqref{eq_bound_D} we obtain
\[ |A_{K_{\lambda,N}}(x) -\lambda| \leq c \left|\frac{\lambda - \lambda_0}{r_N}\right|^{N+1},\]
as we wanted to show.
\end{proof}

\section{Concluding Remarks}
\subsection{Convergence of the Power Series} As mentioned already, the convergence of \eqref{eq_formalpowerseries} in some neighborhood of $\lambda_0$ would finally show an actual counterexample to the Barker--Larman problem in dimension $4$ in the class of analytic, symmetric convex bodies of revolution.
		The key issue is, of course, to obtain a uniform bound $0 < r_0 < r_N$ for all $N$.
		We choose to leave this problem open because it depends on exceedingly complicated computations and estimates on the $N$-th differential of $W$.
		A key issue is that in order to bound $\|W^{N}(h_0)[v_1, \ldots, v_N]\|$, for instance, using the Cauchy formula for
		\[(t_1, \ldots, t_N) \mapsto W(h_0 + t_1 v_1 + \cdots + t_N v_N),\]
		one needs to control the $C^1$ norm of the $v_i$, so that $h_0 + t_1 v_1 + \cdots + t_N v_N$ is invertible.

		\subsection{Choice of the Bifurcation point} The bifurcation point $\lambda_0$ was selected in this paper because the first differential of $W$ is the easiest to describe, in terms of kernel and range.
		More concretely, we compute $W'$ at the function $x \mapsto x - \frac \pi 3$.
		But the same iterative scheme could be applied to other values $\lambda_* = \omega_3 \tan(\alpha)^3$, provided that the matrices $M_n$ in Section \ref{sec_operatorW} remain invertible.
		This would amount to study the first and second differential of $W$ at the function $x \mapsto x-\alpha$.

		Clearly we cannot apply our method unless $W'(x \mapsto x - \alpha)$ has a non-trivial kernel, which is the case of every $\alpha = \frac {a \pi}b$ where $a,b \in \N$ and $b$ is odd.
		We conjecture here that these counterexamples exist around $\lambda_q>0$ where $q = \frac 1\pi \arctan((\lambda_q/\omega_3)^{1/3}) \in \mathbb Q$ can be written with an odd denominator.
		This would complement the fact that the positive results obtained in \cite{haddad2026convex}, concern only values of $\lambda$ for which $\frac 1\pi \arctan((\lambda/\omega_3)^{1/3})$ is difficult to approximate by rational numbers.
		\subsection{Dead ends} This problem can be approached unsuccessfully in many different and creative ways.

		One of them is to try to apply the Crandall--Rabinowitz bifurcation theorem to the non-linear operator that computes the areas of the section of a convex body $K$, as a function of the radial function of $K$. Here one is readily confronted with the small divisors problem, which prevents a diagonal operator $T(e_i) = P^d_n(t_0) e_i$ to be invertible. Here $T: \ell_2 \to \ell_2$, $e_i$ are the canonical vectors, $P^d_n$ are the Legendre-Gegenbauer polynomials, and $t_0$ is the candidate to bifurcation point, which is a zero of $P^d_3$. Curiously enough, $d=2,4$ appear to be the only dimensions where the zeros of $P^d_n$ can have intersection with $P^d_m$ for $n \neq m$. This suggests that the choice of the dimension $d=4$, which is fixed in this paper and in \cite{haddad2026convex}, might be connected to deeper problems in the theory of Legendre polynomials.

		A second dead end is to try to construct the boundary of $L$ as an invariant set of the operator $T_K$ defined in \cite{haddad2026convex}. Numerical experiments suggest that this invariant sets, even though they contain convex curves, are strange-looking fractals, and certainly discard any computer assisted solution, like validated numerics.

\bibliographystyle{abbrv}
\bibliography{../references}

\end{document}